\documentclass{amsart}
\usepackage{amsmath}

\newtheorem{theorem}{Theorem}[section]
\newtheorem{lemma}[theorem]{Lemma}
\newtheorem{prop}[theorem]{Proposition}
\theoremstyle{definition}

\newtheorem{ass}{\textbf{Assumption}}

\theoremstyle{remark}

\numberwithin{equation}{section}

 \allowdisplaybreaks[4]
\begin{document}

\title{Discretization Error of Stochastic Iterated Integrals}


\author{Yuping Song}
\address{Department  of Mathematics, Zhejiang University, Hangzhou, PRC.}
\email{songyuping@shnu.edu.cn}
\thanks{Yuping Song was supported by the General Research Fund of Shanghai Normal
University (No. SK201720).}

\author{Hanchao Wang }
\address{Zhongtai Securities Institute for Financial Studies,  Shandong University,  Jinan,  250100, PRC.}
\email{wanghanchao@sdu.edu.cn}
\thanks{Hanchao Wang is the corresponding author, he  was supported by National Natural Science Foundation of China (No. 11371317, 11526205,  11626247), and  the Fundamental Research Fund of Shandong University (No. 2016GN019). }


\date{}


\begin{abstract}
 In this paper, the weak convergence about the discretization error
of stochastic iterated  integrals in the Skorohod sense are studied,
 while the integrands and integrators of iterated integrals are supposed to be
semimartingales with jumps.  We explored the rate of
convergence of its approximation based on the asymptotic behaviors of the
associated normalized error and obtained that the rate is $1/n$ when
the driving process is semimartingale with a nonvanishing continuous
martingale component. As an application, we also studied the discretization
of  the Dol\'{e}ans-Dade exponential.

  \end{abstract}

\maketitle
\section{Introduction}
Let $(\Omega, \mathcal{F}, (\mathcal{F}_{t})_{t\in [0,1]}, \mathbb{P})$ be a stochastic basis. The present paper studies the weak convergence of a sequence of stochastic processes $X^n=\{X^n_t\}_{t\in [0,1]}$ defined as
\begin{equation}
\label{m}
  X_t^n=\sum_{i=1}^{[nt]}\int_{\frac{i-1}{n}}^{\frac{i}{n}}\int_{\frac{i-1}{n}}^{s-}(Y_{r-}-Y_{\frac{i-1}{n}})dY_rdY_s\end{equation}
in the Skorohod space $\mathbb{D}[0,1]$, where $Y=\{Y_t\}_{t\in [0,1]}$ is semimartingales with jumps on  $(\Omega, \mathcal{F}, (\mathcal{F}_{t})_{t\in [0,1]}, \mathbb{P})$. The stochastic iterated integral $\int_{0}^{t}\int_{0}^{s-}Y_{r-}dY_rdY_s$   is usually defined as a limit of Riemann sums via discretization of $Y$.   We are interested in the asymptotic error distributions of $X$, i.e. the weak convergence of $X$ in   Skorohod space $\mathbb{D}[0,1]$.

Solving the problem can be seen as an extension of previous works on the discretization error of
                                    \begin{equation}
\label{mm}
  \Upsilon_t^n=\sum_{i=1}^{[nt]}\int_{\frac{i-1}{n}}^{\frac{i}{n}}(S_{s-}-S_{\frac{i-1}{n}})dY_s.\end{equation}
   Rootz\'en  \cite{r} studied the  the weak convergence of $\Upsilon$,   when $Y$ is a Brownian motion.   Jacod and Protter \cite{jp} obtained the weak convergence of $\Upsilon$ and asymptotic error distributions of Euler scheme for stochastic differential equation,  when $S=Y$ are semimartingale. Jacod \cite{j} extended \cite{jp} to the pure jump L\'evy processes.  Wang \cite{w} extended Jacod's work to pure jump semimartingales.  In  \cite{jp}, the rate of convergence is  $1/\sqrt{n}$ when the driving process is semimartingale with a nonvanishing continuous martingale component.  The rates of convergence in \cite{j} and \cite{w}  depended  on the concentration of L\'evy measure of underlying processes.   Hayashi and Mykland \cite{hm} discussed this problem  in the financial content for hedging error  when $Y$ is continuous local martingale.  Tankov and Voltchkova \cite{tv} obtained the asymptotic distribution of  hedging error when $Y$ is semimartingale with jumps.  On the other hand, some authors recently consider the random discretization scheme for (\ref{mm}), see \cite{f}, \cite{lr}, \cite{zs} and so on, these papers  on  random discretization scheme   are confined to the case that $Y$ is continuous local martingales.

     So far, the convergence of $X^n$ in (\ref{m}) is rarely studied.  Yan \cite{y} studied this problem when $Y$ is continuous semimartingale,  derived the asymptotic error distribution  of  Milstein scheme for stochastic differential equation, obtained that  the rate of convergence is  $1/n$ when the driving process is continuous semimartingale with a nonvanishing martingale component.   In present work, we focus on the case of $Y$ is semimartingale with jump.

   As an application, we studied the discretization error of Dol\'{e}ans-Dade exponential.  Dol\'{e}ans-Dade exponential plays an important role in the study of weak convergence for semimartingales and has explicit computation of density process. It can be defined through
            \begin{equation}
\label{dde}
    dX_t=X_{t-}dY_t,
\end{equation}
 $X$ is called the  Dol\'{e}ans-Dade exponential  of $Y$.  If  $Y$ is a semimartingale with jumps,  then
          $$X_t=:\mathcal{E}(Y)_t=\exp\{Y_t-Y_0-\frac{1}{2}<Y^c>_t\}\prod_{s\le t}(1+\Delta Y_s)e^{-\Delta Y_s}$$
   where $Y^c$ is continuous local martingale part of $Y$, $<Y^c>$ stands for its  predictable quadratic variation.   If we conduct a  Milstein type scheme for (\ref{dde}), we obtain its asymptotic error  distribution . It can be held as asymptotic error  distribution
of  the Dol\'{e}ans-Dade exponential.

   This paper is organized as follows. We express the main result in Section 2.  In Section 3, some technical lemmas and the proof of main result are presented. The application and discuss will be collected in Section 4.

   \section{Main result}
  A semimartingale $M$ is an {\em It\^{o} semimartingale }on some filtered space $(\Omega,
\mathcal {F}, (\mathcal {F}_{t})_{t\ge 0}, \mathbb{P})$ if  its
characteristics $(B^{M},C^{M},\nu^{M})$ are absolutely continuous
with respect to Lebesgue measure. In other words, the
characteristics of $M$ have the form
\begin{equation}B_{t}^{M}=\int_{0}^{t}b_{s}^{M}ds,~~C_{t}^{M}=\int_{0}^{t}c_{s}^{M}ds,~~\nu^{M}(dt,dx)=dtF_{t}^{M}(dx).
\label{spochr}\end{equation} Here $b^{M}$ and $c^{M}$ are optional
processes, with $c^{M}\ge 0$, and $F^{M}$ is an optional random
measure on $\mathbb{R}$. The triple $(b_{t}^{M}, c_{t}^{M},
F^{M}_{t})$ constitutes the {\em spot characteristics} of $M$. The
details of these concepts and notions can be found in Jacod and
Shiryaev \cite{js}.
 \begin{ass}\label{a1} We assume that $Y$ in (\ref{m}) is stochastic integral driven by a  L\'{e}vy process, that is
                $$Y_t=\int_{0}^{t}\sigma_{s-}dZ_s$$
 where

 (a) $Z$ is L\'{e}vy process with characteristics $(b,c,F)$, where $b\in \mathbb{R}$, $c>0$ and $F$ is a positive measure on $\mathbb{R}$ with $F(\{0\})=0$, $F(\mathbb R)<\infty$ and$\int (x^2\wedge 1)F(dx)<\infty$.

(b) The process $\sigma$ is an It\^{o} semimartingale with spot
characteristics $(b_{t}^{\sigma},c_{t}^{\sigma},F_{t}^{\sigma})$,
which are such that the processes $b_{t}^{\sigma},c_{t}^{\sigma}$
and $\int (x^{2}\wedge 1)F_{t}^{\sigma}(dx)$ are locally bounded.
 \end{ass}

    In fact, if  we denote by $\mu$ the jump random measure of $Z$, and  set $\nu(dt,dx)=dt\times F(dx)$,  $Z$ has the form (see \cite{js})
       $$Z_t=bt+Z_{t}^{c}+x1_{\{|x|\le 1\}}*(\mu-\nu)_t+x1_{\{|x|>1\}}*\mu.$$
   The limiting process in our main results is described in the following.

Set\\
1. $M$ is standard Brownian motion;\\
2. $(N_n')_{n\ge 1}$ and $(N_n'')_{n\ge 1}$ are two sequences of standard normal variables;\\
3. $(K_n')_{n\ge 1}$ and $(K_n'')_{n\ge 1}$ are two sequences of identical distributed random variables, and $K_n$ has same distribution with $N_n^2-1$;\\
4.  $(\xi_n)_{n\ge 1}$ is a sequence of uniform variables on $(0,1)$.

$M$,  $(N_n')_{n\ge 1}$,  $(N_n'')_{n\ge 1}$,  $(K_n')_{n\ge 1}$, $(K_n'')_{n\ge 1}$ and  $(\xi_n)_{n\ge 1}$ are independent from each other and all other random elements.

 Let us denote by $(T_{n})_{n\ge 1}$  arbitrary ordering of all jump times of $Z$,  consisting of stopping times taking value in $(0,1]$.

Now, we present our main results.
 \begin{theorem} \label{mr1}

 Under Assumption \ref{a1}, we have  the following:

 (a) If $c=0$, $nX^n$ weakly converge to 0.

 (b) If $c>0$, $nX^n$ weakly converge to $X$, where
  \begin{eqnarray*}
                 X_s =\frac{\sqrt{6c^3}}{6}\int_0^s\sigma_{t-}^{3}dM_t-c\sum_{n:T_n\le t}[\sqrt{\xi_n}K_{n}'+\sqrt{c(1-\xi_n)\xi_n}N_n'N_{n}''+\sqrt{1-\xi_n}K_{n}'')]\sigma_{T_{n}-}^{3}\Delta Z_{T_n}\end{eqnarray*}

\end{theorem}

 \section{Preliminaries and the proof of Theorem \ref{mr1}}

 \subsection{Localization}

We first  reduce the problem to a situation where $Y$
satisfies some strengthened versions of our assumptions, which are
as follows.

\begin{ass} \label{a2} We have Assumption \ref{a1}, and moreover

(a) we have $|\Delta Z_{t}|$ and  $|\sigma_{t}|$ are  bounded for
all $t\in[0,1]$;

(b)the processes $|b_{t}^{\sigma}|$, $c_{t}^{\sigma}$ and $\int (x^{2}\wedge
1)F_{t}^{\sigma}(dx)$are bounded.

\end{ass}

\begin{lemma} \label{local}In Theorem \ref{mr1}, one can replace Assumption 1   with Assumption \ref{a2} . \end{lemma}
\begin{proof}

We suppose Theorem  \ref{mr1} holds under the Assumptions \ref{a2},  we need to prove that Theorem \ref{mr1} still holds under Assumption \ref{a1}.

Consider the new L\'{e}vy process $Z(p)_{t}=Z_{t}-\sum_{s\le
t}\Delta Z_{s}1_{\{|\Delta Z_{s}|\ge p\}}$. Association with $Z(p)$
the same term as $Z$ , write  $Y(p)$, $X(p)$ instead of $ Y, X^{n}$,
respectively.

By hypothesis, we have a sequence of stopping times $(\tau_{p})_{p\ge 1}$, and a sequence of non-random time $(t_{p})_{p\ge 1}$, with $\tau_{p}\le t_{p}$, and $\tau_{p}\uparrow\infty$ as $p\rightarrow\infty$ such that
     $$|b_{t}^{\sigma}|\le p, c_{t}^{\sigma}\le p, |\sigma_{t}|\le p, \int(x^{2}\wedge 1)F_{t}^{\sigma}(dx)\le p, |\Delta Z|\le p$$
when $t\le \tau_{p}$.

Thus, set $\sigma(p)_{t}=\sigma_{t\wedge \tau_{p}}$, $Y(p)_{t}=\int_{0}^{t}\sigma(p)_{s-}dZ(p)_{s}$. We easily obtain
      $$t\le \tau_{p}\Rightarrow \sigma(p)_{t}=\sigma, Z(p)_{t}=Z,  Y(p)=Y.$$
Our hypothesis now implied $nX^{n}(p)$ converges in law to $X(p)$ with $\sigma(p)$ instead of $\sigma$.
We see the restriction of $X$ to $[0,\tau_{p})$ is a version of the restriction of $X(p)$ to $[0,\tau_{p})$.

For any continuous bounded function $\Phi_{t}$ on the Skorokhod space $\mathbb{D}([0,1],\mathbb{R})$ which depends on the sample path only up to time $t$, we have
   $$|\mathbb{E}(\Phi_{t}(nX^{n}))-\mathbb{E}(\Phi_{t}(nX^{n}(p)))|\le 2||\Phi_{t}||\mathbb{P}(\tau_{p}\le t),$$
$$|\mathbb{E}(\Phi_{t}( X))-\mathbb{E}(\Phi_{t}(X(p)))|\le 2||\Phi_{t}||\mathbb{P}(\tau_{p}\le t).$$
Since $\mathbb{P}(\tau_{p}\le t)\rightarrow 0$ and
         $$\mathbb{E}(\Phi_{t}(nX^{n}(p)))\rightarrow \mathbb{E}(\Phi_{t}(X(p)))$$
for every $t$ as $p\rightarrow\infty$, we obtain this lemma. \end{proof}

\subsection{Preliminaries}
 In this paper, we will be dealing with the weak convergence of stochastic integral  in the Skorohod topology.  We denote by $\Rightarrow$ the weak convergence for this topology.

 We first recall some facts on convergence of stochastic integrals,  which are from Kurtz and Protter \cite{kp}.

 First recall that, for every $\delta>0$, any semimartingale can be written as
       $$X_t=X_0+A_{X}(\delta)_t+M_{X}(\delta)_{t}+\sum_{s\le t}\Delta X_s1_{\{|\Delta X_{s}|>\delta\}}$$
      where $A_{X}(\delta)$ is a predictable process with finite variation, null at 0, $M_{X}(\delta)$ is a local martingale null at 0, and $\Delta M(\delta)_t\le \delta$.
\begin{theorem} \label{kp1}(Kurtz and Protter \cite{kp}.)     Let $X^n$ be a sequence of semimartingales, $H^n$ a sequenrce of predictable processes. If there exist a predictable process $H$, and semimartingale $X$ such that
     \begin{equation}\label{j1c}
            \sup_{t\in[0,1]}|H^n_t-H_t|\stackrel{\mathbb{P}}{\rightarrow} 0,     \end{equation}
       \begin{equation}\label{j2c}              X^n\Rightarrow X ,  \end{equation}
 and
     $$<M_{X^n}(\delta),M_{X^n}(\delta)>_1+\int_0^1|dA_{X^n}(\delta)_s|+\sum_{0\le s\le 1}X^n_s1_{\{|\Delta X^n_{s}|>\delta\}}$$
     is tight for some $\delta>0$. Then we have
         $$(X^n,H_-^n\cdot X^n)\Rightarrow (X, H_-\cdot X).$$
     \end{theorem}

  Next, we will recall some facts about the stable convergence.
Firstly, we will mention the definition of the stable convergence in
law and its property, secondly, we will present  limit theorem for
partial sums of triangular arrays of random variables, one can refer
to Jacod \cite{j3} \cite{j4} \cite{ja} or Jacod and Shiryaev
\cite{js} for more details.

\noindent {\textbf{1) Stable convergence in law.}}

This notation was firstly introduced by R\'{e}nyi \cite{re}, which
is exposited by Aldous and Eagleson \cite{ald}.

A sequence of random variables $Z_{n}$ defined on the probability
space $(\Omega, \mathcal {F}, \mathbb{P}),$ taking their values in
the state space $(E, \mathcal {E}),$ assumed to be Polish. We say
that $Z_{n}$ stably converges in law if there is a probability
measure $\eta$ on the product $(\Omega \times E, \mathcal {F} \times
\mathcal {E}),$ such that $\eta(A \times E) = \mathbb{P}(A)$ for all
$A \in \mathcal {F}$ and
\begin{equation}
\label{def1} \mathbb{E}(Y f(Z_{n})) \longrightarrow \int{Y(\omega)
f(x) \eta(d\omega, dx)}
\end{equation}
for all bounded continuous functions $f$ on $E$ and bounded random
variables $Y$ on $(\Omega, \mathcal {F}).$

Take $\tilde{\Omega} = \Omega \times E$, $\tilde{\mathcal {F}} =
\mathcal {F} \times \mathcal {E}$ and endow $(\tilde{\Omega},
\tilde{\mathcal {F}})$ with the probability $\eta,$ and put
$Z(\omega, x) = x,$ on the extension $(\tilde{\Omega},
\tilde{\mathcal {F}}, \tilde{\mathbb{P}})$ of $({\Omega}, {\mathcal
{F}}, {\mathbb{P}})$ with the expectation $\tilde{\mathbb{E}}$ we
have
\begin{equation}
\label{def2} \mathbb{E}(Y f(Z_{n})) \longrightarrow
\tilde{\mathbb{E}}(Y f(Z)),
\end{equation}
then we say that $Z_{n}$ converges stably to $Z,$ denoted by
$\stackrel{stably} \Rightarrow.$

The stable convergence implies the following crucial property, which
is fundamental for the proof of lemma \ref{3.5}.
\begin{prop}
\label{pro1} if $Z_{n} \stackrel{stably} \Rightarrow Z $ and if
$Y_{n}$ and $Y$ are variables defined on $(\Omega, \mathcal {F},
\mathbb{P})$ and with values in the same Polish space F, then
\begin{equation}
\label{def3} Y_{n} \stackrel{\mathbb{P}} \longrightarrow
Y~~~~~\Rightarrow~~~~~(Y_{n}, ~Z_{n}) \stackrel{stably} \Rightarrow
(Y, ~Z),
\end{equation}
which implies that $Y_{n} + Z_{n} \stackrel{stably} \Rightarrow Y +
Z$ through the continuous function $g(x, y) = x + y.$
\end{prop}

\noindent {\textbf{2) Convergence of triangular arrays.}}

In this part, we give the available convergence criteria for stable
convergence of partial sums of triangular arrays.

\begin{theorem}
\label{theo1} (Jacod's stable convergence theorem \cite{ja}.) A
sequence of $\mathbb{R}-$valued variables $(\zeta_{n, i}: i \geq 1)$
defined on the filtered probability space $(\Omega, \mathcal {F},
(\mathcal {F})_{t \geq 0}, \mathbb{P})$ is $\mathcal {F}_{i
\Delta_{n}}-$measurable for all $n, i.$ Assume there exists a
continuous adapted $\mathbb{R}-$valued process of finite variation
$B_{t}$ and a continuous adapted and increasing process $C_{t}$, for
any $t
> 0,$ we have
\begin{equation}
\label{def4} \sup_{0 \leq s \leq t}\big|\sum_{i =
1}^{[s/\Delta_{n}]}\mathbb{E}\big[\zeta_{n, i} | \mathcal {F}_{(i -
1)\Delta_{n}}\big] - B_{s}\big| \stackrel{\mathbb{P}} \longrightarrow 0,
\end{equation}
\begin{equation}
\label{def5} ~~~~~~~~~~~~~~~~~~~~~~~~~\sum_{i =
1}^{[t/\Delta_{n}]}\big(\mathbb{E}\big[\zeta_{n, i}^{2} | \mathcal
{F}_{(i - 1)\Delta_{n}}\big] - \mathbb{E}^{2}\big[\zeta_{n, i} |
\mathcal {F}_{(i - 1)\Delta_{n}}\big]\big) - C_{t} \stackrel{\mathbb{P}}
\longrightarrow 0,
\end{equation}
\begin{equation}
\label{def6} \sum_{i = 1}^{[t/\Delta_{n}]}\mathbb{E}\big[\zeta_{n,
i}^{4} | \mathcal {F}_{(i - 1)\Delta_{n}}\big] \stackrel{\mathbb{P}}
\longrightarrow 0.
\end{equation}
Assume also
\begin{equation}
\label{def7} \sum_{i = 1}^{[t/\Delta_{n}]}\mathbb{E}\big[\zeta_{n,
i} \Delta_{n}^{i}H | \mathcal {F}_{(i - 1)\Delta_{n}}\big]
\stackrel{\mathbb{P}} \longrightarrow 0,
\end{equation}
where either H is one of the components of Wiener process $W$ or is
any bounded martingale orthogonal (in the martingale sense) to $W$
and $\Delta_{n}^{i}H = H_{i\Delta_{n}} - H_{(i - 1)\Delta_{n}}.$

Then the processes $$\sum_{i = 1}^{[t/\Delta_{n}]}\zeta_{n, i}
\stackrel{stably} \Rightarrow B_{t} + M_{t},$$ where $M_{t}$ is a
continuous process defined on an extension $\big(\widetilde{\Omega},
\widetilde{\mathcal {F}},  \widetilde{\mathbb{P}}\big)$ of the filtered
probability space $\big({\Omega}, {\mathcal {F}}, \mathbb{P}\big)$ and
which, conditionally on the the $\sigma-$filter $\mathcal {F}$, is a
centered Gaussian $\mathbb{R}-$valued process with
$\widetilde{\mathbb{E}}\big[M_{t}^{2} | \mathcal {F}\big] = C_{t}.$
\end{theorem}

\subsection{Asymptotic properties on L\'{e}vy process}

In this section, we study asymptotic properties  of

 $$S^{n}_{t}:=\sum_{i=1}^{[nt]}\int_{(i-1)/n}^{i/n}\int_{(i-1)/n}^{s-}(Z_{r-}-Z_{(i-1)/n})dZ_{r}dZ_{s}.$$

           For $\varepsilon>0$, we set
          $$M^{\varepsilon}=x1_{\{|x|\le \varepsilon\}}*(\mu-\nu),~~~~N^{\varepsilon}=x1_{\{|x|\ge \varepsilon\}}*(\mu-\nu),$$
          $$A^{\varepsilon}=x1_{\{|x|\ge \varepsilon\}}*\mu,~~~~b_{\varepsilon}=b-\int_{\varepsilon<|x|\le 1}xF(dx).$$
           Obviously,
          $$Z_t=Z_t^c+M_t^{\varepsilon}+A_t^{\varepsilon}+b_{\varepsilon}t$$
    From now on, we fix $\varepsilon>0$, and denote by $0<T_1<T_2<\cdots<T_n<\cdots$  the successive jump times of $Z$ with size bigger than $\varepsilon$.    We define
       $$T_{+}(n,i)=\inf\{\frac{k}{n}:k\ge 1, \frac{k}{n}\ge T_{i}\},~~~~T_{-}(n,i)=T_{+}(n,i)-\frac{1}{n};$$
Set
      $$\alpha_j^n=n\Delta Z_{T_j}(Z_{T_j}^{c}-Z_{T_-(n,j)}^{c})(Z_{T_+(n,j)}^{c}-Z_{T_j}^{c}),$$
      $$\beta_j^n=n\Delta Z_{T_j}[\frac{1}{2}(Z_{T_j}^{c}-Z_{T_-(n,j)}^{c})^2-\frac{1}{2}(T_j-T_-(n,j))],$$
         $$\gamma_j^n=n\Delta Z_{T_j}[\frac{1}{2}(Z_{T_+(n,j)}^{c}-Z_{T_j}^{c})^2-\frac{1}{2}(T_+(n,j)-T_j)].$$

           \begin{lemma}\label{3.5}
  If we denote $W_t=Z^c_t+b_\varepsilon t := L_{t} +
 F_{t},$
      $$M^{n,\varepsilon}_t=\sum_{i=1}^{[nt]}\int_{(i-1)/n}^{i/n}\int_{(i-1)/n}^{s-}(W_{r-}-W_{(i-1)/n})dW_{r}dW_{s}.$$
     We have
         $$nM^{n,\varepsilon}\Rightarrow M$$
 where $M$ are  continuous local martingales, which are independent of $Z^c$.

    \end{lemma}

 \begin{proof}

 \begin{eqnarray*}
M^{n,\varepsilon}_{t} & = &
\sum_{i=1}^{[nt]}\int_{(i-1)/n}^{i/n}\int_{(i-1)/n}^{s}(W_{r}-W_{(i-1)/n})dW_{r}dW_{s}\\
& = &
\sum_{i=1}^{[nt]}\int_{(i-1)/n}^{i/n}\int_{(i-1)/n}^{s}(L_{r}-L_{(i-1)/n})dL_{r}dL_{s}
+
\sum_{i=1}^{[nt]}\int_{(i-1)/n}^{i/n}\int_{(i-1)/n}^{s}(L_{r}-L_{(i-1)/n})dL_{r}dF_{s}\\
& = &
\sum_{i=1}^{[nt]}\int_{(i-1)/n}^{i/n}\int_{(i-1)/n}^{s}(L_{r}-L_{(i-1)/n})dF_{r}dW_{s}
+
\sum_{i=1}^{[nt]}\int_{(i-1)/n}^{i/n}\int_{(i-1)/n}^{s}(F_{r}-F_{(i-1)/n})dL_{r}dW_{s}\\
& = &
\sum_{i=1}^{[nt]}\int_{(i-1)/n}^{i/n}\int_{(i-1)/n}^{s}(F_{r}-F_{(i-1)/n})dF_{r}dW_{s}\\
& := & M^{n,\varepsilon}_{1, t} + M^{n,\varepsilon}_{2, t} +
M^{n,\varepsilon}_{3, t} + M^{n,\varepsilon}_{4, t} +
M^{n,\varepsilon}_{5, t}.
\end{eqnarray*}
The results of $$n M^{n,\varepsilon}_{2, t} \stackrel {L^{2}}
\rightarrow 0,~~n M^{n,\varepsilon}_{3, t} \stackrel {L^{2}}
\rightarrow 0,~~n M^{n,\varepsilon}_{4, t} \stackrel {L^{2}}
\rightarrow 0,~~n M^{n,\varepsilon}_{5, t} \stackrel {L^{2}}
\rightarrow 0$$ are easily obtained according to lemma 7.7 in Yan
\cite{y}, so we get $$n M^{n,\varepsilon}_{2, t} \stackrel {\mathbb{P}}
\rightarrow 0,~~n M^{n,\varepsilon}_{3, t} \stackrel {\mathbb{P}} \rightarrow
0,~~n M^{n,\varepsilon}_{4, t} \stackrel {\mathbb{P}} \rightarrow 0,~~n
M^{n,\varepsilon}_{5, t} \stackrel {\mathbb{P}} \rightarrow 0.$$ Next, we
will employ Theorem \ref{theo1} to the part $n M^{n,\varepsilon}_{1,
t}.$
$$n M^{n,\varepsilon}_{1,
t} = n
\sum_{i=1}^{[nt]}\int_{(i-1)/n}^{i/n}\int_{(i-1)/n}^{s}(L_{r}-L_{(i-1)/n})dL_{r}dL_{s}
:= \sum_{i=1}^{[nt]} q_{i},$$ where $q_{i} = n
\int_{(i-1)/n}^{i/n}\int_{(i-1)/n}^{s}(L_{r}-L_{(i-1)/n})dL_{r}dL_{s}.$

\noindent {\textbf{(1)}} $\sum_{i=1}^{[nt]} \mathbb{E}_{i - 1}
[q_{i}] \equiv 0$ by the martingale property of stochastic integral.
\bigskip

\noindent {\textbf{(2)}}  Since $L_{t}$ is a Brownian motion,
we have
\begin{eqnarray*}
\sum_{i=1}^{[nt]} \mathbb{E}_{i - 1} [q^{2}_{i}] & = & n^{2}
\sum_{i=1}^{[nt]} \mathbb{E}_{i - 1}
\left[\int_{(i-1)/n}^{i/n}\int_{(i-1)/n}^{s}(L_{r}-L_{(i-1)/n})dL_{r}dL_{s}\right]^{2}\\
& = & c \cdot n^{2} \sum_{i=1}^{[nt]} \mathbb{E}_{i - 1}
\int_{(i-1)/n}^{i/n}\left[\int_{(i-1)/n}^{s}(L_{r}-L_{(i-1)/n})dL_{r}\right]^{2}ds\\
& = & c \cdot n^{2} \sum_{i=1}^{[nt]}
\int_{(i-1)/n}^{i/n}\mathbb{E}_{i -
1}\left[\int_{(i-1)/n}^{s}(L_{r}-L_{(i-1)/n})dL_{r}\right]^{2}ds\\
& = & c^{2} \cdot n^{2} \sum_{i=1}^{[nt]}
\int_{(i-1)/n}^{i/n}\int_{(i-1)/n}^{s}\mathbb{E}_{i -
1}\left[L_{r}-L_{(i-1)/n}\right]^{2}dr ds\\
& = & c^{3} \cdot n^{2} \sum_{i=1}^{[nt]}
\int_{(i-1)/n}^{i/n}\int_{(i-1)/n}^{s}\left[r - (i-1)/n\right] dr
ds\\
& = & \frac{c^{3}}{6} t,
\end{eqnarray*}
by  Fubini's Theorem in the third equation.

\noindent {\textbf{(3)}}  Using BDG and H\"{o}lder inequality, we
have
\begin{eqnarray*}
\sum_{i=1}^{[nt]} \mathbb{E}_{i - 1} [q^{4}_{i}] & = & n^{4}
\sum_{i=1}^{[nt]} \mathbb{E}_{i - 1}
\left[\int_{(i-1)/n}^{i/n}\int_{(i-1)/n}^{s}(L_{r}-L_{(i-1)/n})dL_{r}dL_{s}\right]^{4}\\
& \leq & n^{4} \sum_{i=1}^{[nt]} \mathbb{E}_{i - 1}
\left[\int_{(i-1)/n}^{i/n}\left(\int_{(i-1)/n}^{s}(L_{r}-L_{(i-1)/n})dL_{r}\right)^{2}
c_{s} ds\right]^{2}\\
& \leq & n^{4} \sum_{i=1}^{[nt]} \mathbb{E}_{i - 1}
\int_{(i-1)/n}^{i/n}\left(\int_{(i-1)/n}^{s}(L_{r}-L_{(i-1)/n})dL_{r}\right)^{4}
ds \int_{(i-1)/n}^{i/n} c^{2}_{s} ds\\
& \leq & K n^{3} \sum_{i=1}^{[nt]} \int_{(i-1)/n}^{i/n}\mathbb{E}_{i
- 1}\left(\int_{(i-1)/n}^{s}(L_{r}-L_{(i-1)/n})^{2} c_{r}
dr\right)^{2} ds\\
& \leq & K n^{3} \sum_{i=1}^{[nt]} \int_{(i-1)/n}^{i/n}\mathbb{E}_{i
- 1}\left(\int_{(i-1)/n}^{s}(L_{r}-L_{(i-1)/n})^{4} dr
\int_{(i-1)/n}^{s}c^{2}_{r} dr \right) ds\\
& \leq & K^{2} n^{2} \sum_{i=1}^{[nt]}
\int_{(i-1)/n}^{i/n}\int_{(i-1)/n}^{s}\mathbb{E}_{i -
1}(L_{r}-L_{(i-1)/n})^{4} dr ds\\
& = & O\left(\frac{1}{n}\right) \rightarrow 0.
\end{eqnarray*}

\noindent {\textbf{(4)}} If $H$ is orthogonal to $W$, we have
$$\sum_{i=1}^{[nt]} \mathbb{E}_{i -
1} [q_{i}\Delta_{i}H] \equiv 0.$$ If $H = W,$ then by the integration by parts
formula,
\begin{eqnarray*}
\sum_{i=1}^{[nt]} \mathbb{E}_{i - 1} [q_{i}\Delta_{i}H] & = & n
\sum_{i=1}^{[nt]} \mathbb{E}_{i - 1} \left[
\int_{(i-1)/n}^{i/n}\int_{(i-1)/n}^{s}(L_{r}-L_{(i-1)/n})dL_{r}dL_{s}
\int_{(i-1)/n}^{i/n} d W_{s} \right]\\
& = & n \sum_{i=1}^{[nt]} \mathbb{E}_{i -
1}\left[\int_{(i-1)/n}^{i/n}\int_{(i-1)/n}^{t}\int_{(i-1)/n}^{s}(L_{r}-L_{(i-1)/n})dL_{r}dL_{s}d
W_{t}\right]\\ & ~ & + n \sum_{i=1}^{[nt]} \mathbb{E}_{i -
1}\left[\int_{(i-1)/n}^{i/n}\int_{(i-1)/n}^{t}W_{s}ds\int_{(i-1)/n}^{t}(L_{r}-L_{(i-1)/n})dL_{r}dL_{t}\right]\\
& ~ & + n \sum_{i=1}^{[nt]} \mathbb{E}_{i -
1}\left[\int_{(i-1)/n}^{i/n}\int_{(i-1)/n}^{t}(L_{r}-L_{(i-1)/n})dL_{r}\sqrt{c}dt\right]\\
& = & 0,
\end{eqnarray*}
by the Fubini's Theorem for the third part and the martingale
property of stochastic integral.

Based on Proposition \ref{pro1} and the above calculations, we get
$$nM^{n,\varepsilon}\stackrel{stably}\Rightarrow M$$
where $M$ are  continuous local martingales, which are independent
of $Z^c$.
\end{proof}

   \begin{lemma}\label{3.6}
   We have $$(W, \alpha^n, \beta^n, \gamma^n)\stackrel{stably}{\Rightarrow} (W, \alpha,\beta,\gamma),$$
   where $W$ is a Browinian motion,  $\alpha^n=(\alpha_j^n)_{j\ge 1}$,  $\beta^n=(\beta_j^n)_{j\ge 1}$, $\gamma^n=(\gamma_j^n)_{j\ge 1}$, $\alpha=(\alpha_j)_{j\ge 1}$,  $\beta=(\beta_j)_{j\ge 1}$, $\gamma=(\gamma_j)_{j\ge 1}$,
    $$\alpha_j=\sqrt{c(1-\xi_j)\xi_j}N_j'N_{j}''\Delta Z_{T_j},$$
      $$\beta_j=\frac{1}{2}\sqrt{c\xi_j}K_j'\Delta Z_{T_j},~~\gamma_j=\frac{1}{2}\sqrt{c(1-\xi_j)}K_j''\Delta Z_{T_j},$$
       $$\eta_j=\sqrt{c\xi_j}(N_j')^2\Delta Z_{T_j},~~\kappa_j=\sqrt{c(1-\xi_j)}N_j''(\Delta Z_{T_j})^2.$$

   \end{lemma}

    \begin{proof}

    In order to prove the result, we need to show that
     $$\mathbb{E}(h(Z)g( W,\alpha^n, \beta^n, \gamma^n))\rightarrow\mathbb{E}(h(Z)g(W, \alpha,\beta,\gamma))$$
  for all bounded functions $h$ and uniformly continuous bounded functions $g$.  By the density argument from Jacod and Protter \cite{jp}, it is enough to prove this when $h(Z)=u(Z')v(Z^c)w(A^{\varepsilon})$, where $Z'=Z-Z^{c}-A^{\varepsilon}$.  By the similar argument of Lemma 6.2 in Jacod and Protter \cite{jp}, it remains to prove that

\begin{eqnarray*}
& &\mathbb{E}(u(Z')v(Z^c)w(T_j,  \Delta Z_{T_j})_{1\le j\le k}g( W,\alpha_j^n, \beta_j^n, \gamma_j^n )_{1\le j\le k}1_{\Omega_n})\\
& &~\rightarrow  \mathbb{E}(u(Z')v(Z^c)w(T_j,  \Delta Z_{T_j})_{1\le j\le k}g( W,\alpha_j, \beta_j, \gamma_j)_{1\le j\le k})
\end{eqnarray*}

 where the set $\Omega_n$ on which each interval $((i-1)/n,i/n]$  contains at most one $T_j$ tends to $\Omega$.

    Since the independence and stationary of the increments of the L\'evy process,   it is enough to study the limit of
           \begin{eqnarray*}
& &\mathbb{E}(u(Z')v(Z^c)w(T_j,  \Delta Z_{T_j})_{1\le j\le k}g( W,\sqrt{c(1-\xi_j^n)\xi_j^n}N_j'N_{j}''\Delta Z_{T_j}, \\& &~~\frac{1}{2}\sqrt{c\xi_j^n}K_j'\Delta Z_{T_j},\frac{1}{2}\sqrt{c(1-\xi_j^n)}K_j''\Delta Z_{T_j})_{1\le j\le k}1_{\Omega_n}).
\end{eqnarray*}

If $F_k$ and $G_k$ denote the distribution of $( \Delta Z_{T_j})_{1\le j\le k}$ and $(T_j)_{1\le j\le k}$, the previous expression becomes
    \begin{eqnarray*}
& &\int F_k(dx_1,\cdots,dx_k)G_k(dt_1,\cdots,dt_k)1_{\bigcap_{1\le i\le k}\{[nt_i]\le [nt_{i+1}]\}}
\\
& &\times \mathbb{E}(u(Z')v(Z^c)w(t_j, x_j)_{1\le j\le k}g( W,\sqrt{c(1-\xi_j^n)\xi_j^n}N_j'N_{j}''x_j, \frac{1}{2}\sqrt{c\xi_j^n}K_j'x_j,\\
& &~~\frac{1}{2}\sqrt{c(1-\xi_j^n)}K_j''x_j)_{1\le j\le k}1_{\Omega_n}),
\end{eqnarray*}
   where $\xi_j^n=n(T_j-T_{-}(n,j))$. Since $T_1,T_2,\cdots,T_k$ is independent of $Z$, from the Jacod and Protter \cite{jp}, we have $(\xi_j^n)_{1\le j\le k } \stackrel{stably}{\Rightarrow} (\xi_j)_{1\le j\le k } $ , we complete this proof.  \end{proof}

        \begin{lemma}\label{3.7}
 If we denote
      $$F^{n,\varepsilon,1}_t=\sum_{i=1}^{[nt]}\int_{(i-1)/n}^{i/n}\int_{(i-1)/n}^{s-}(Z^{c}_{r-}-Z^{c}_{(i-1)/n})dA^{\varepsilon}_{r}dZ^{c}_{s},$$
       $$F^{n,\varepsilon,2}_t=\sum_{i=1}^{[nt]}\int_{(i-1)/n}^{i/n}\int_{(i-1)/n}^{s-}(Z^{c}_{r-}-Z^{c}_{(i-1)/n})dZ^{c}_{r}dA^{\varepsilon}_{s},$$
       $$F^{n,\varepsilon,3}_t=\sum_{i=1}^{[nt]}\int_{(i-1)/n}^{i/n}\int_{(i-1)/n}^{s-}(A^{\varepsilon}_{r-}-A^{\varepsilon}_{(i-1)/n})dZ^{c}_{r}dZ^{c}_{s},$$
        $$F^{n,\varepsilon,4}_t=\sum_{i=1}^{[nt]}\int_{(i-1)/n}^{i/n}\int_{(i-1)/n}^{s-}(A^{\varepsilon}_{r-}-A^{\varepsilon}_{(i-1)/n})dZ^{c}_{r}dA^{\varepsilon}_{s},$$
        $$F^{n,\varepsilon,5}_t=\sum_{i=1}^{[nt]}\int_{(i-1)/n}^{i/n}\int_{(i-1)/n}^{s-}(A^{\varepsilon}_{r-}-A^{\varepsilon}_{(i-1)/n})dA^{\varepsilon}_{r}dZ^{c}_{s},$$
 $$F^{n,\varepsilon,6}_t=\sum_{i=1}^{[nt]}\int_{(i-1)/n}^{i/n}\int_{(i-1)/n}^{s-}(Z^{c}_{r-}-Z^{c}_{(i-1)/n})dA^{\varepsilon}_{r}dA^{\varepsilon}_{s},$$
    $$F^{n,\varepsilon}=\sum_{i=1}^{6}F^{n,\varepsilon,i},$$
we have
         $$nF^{n,\varepsilon} \Rightarrow F^\varepsilon ,$$
 where
       $$F^\varepsilon=\sum_{j\ge1}(\sqrt{c(1-\xi_j)\xi_j}N_j'N_{j}''\Delta Z_{T_j}+\frac{1}{2}\sqrt{c\xi_j}K_j'\Delta Z_{T_j}+\frac{1}{2}\sqrt{c(1-\xi_j)}K_j''\Delta Z_{T_j})1_{[T_j,1]}(t).$$

    \end{lemma}

     \begin{proof}

     By the definition of It\^o type stochastic integral for L\'evy processes,
     we can obtain
     $$\int_{(i-1)/n}^{i/n}\int_{(i-1)/n}^{s-}(Z^{c}_{r-}-Z^{c}_{(i-1)/n})dA^{\varepsilon}_{r}dZ^{c}_{s}=\alpha^n_i,$$
     $$\int_{(i-1)/n}^{i/n}\int_{(i-1)/n}^{s-}(Z^{c}_{r-}-Z^{c}_{(i-1)/n})dZ^{c}_{r}dA^{\varepsilon}_{s}=\beta^n_i,$$
    $$\int_{(i-1)/n}^{i/n}\int_{(i-1)/n}^{s-}(A^{\varepsilon}_{r-}-A^{\varepsilon}_{(i-1)/n})dZ^{c}_{r}dZ^{c}_{s}=\gamma^n_i$$
     $$F^{n,\varepsilon,4}=F^{n,\varepsilon,5}=F^{n,\varepsilon,6}=0$$
    on $\Omega_n$, on which  each interval $((i-1)/n,i/n]$  contains at most one $T_j$ tends to $\Omega$.
   By Lemma \ref{3.6}, this lemma is proved.
     \end{proof}
         \begin{lemma}\label{3.8}
 If we denote
      $$K^{n,\varepsilon}_t=S^{n}-M^{n,\varepsilon}_t-F^{n,\varepsilon}_t,$$
we have
       $$\lim_{\varepsilon\rightarrow 0}\limsup_{n}\mathbb{P}(\sup_{t\in [0,1]}|nK_{t}^{n,\varepsilon}|>\rho)=0,$$
     for every $\rho>0$.  \end{lemma}

\begin{proof}

We first consider $K_{t}^{n,\varepsilon}$.  In fact, $K_{t}^{n,\varepsilon}$ can be divided into the following parts:
 \begin{equation}
\label{3.81}
J^{n,\varepsilon,1}:=\sum_{i=1}^{[nt]}\int_{(i-1)/n}^{i/n}\int_{(i-1)/n}^{s-}(Z_{r-}-Z_{(i-1)/n})dZ_{r}dM^{\varepsilon}_s,\end{equation}
 \begin{equation}
\label{3.82}
J^{n,\varepsilon,2}:=\sum_{i=1}^{[nt]}\int_{(i-1)/n}^{i/n}\int_{(i-1)/n}^{s-}(M^{\varepsilon}_{r-}-M^{\varepsilon}_{(i-1)/n})dW_{r}dW_s,\end{equation}
 \begin{equation}
\label{3.83}
J^{n,\varepsilon,3}:=\sum_{i=1}^{[nt]}\int_{(i-1)/n}^{i/n}\int_{(i-1)/n}^{s-}(M^{\varepsilon}_{r-}-M^{\varepsilon}_{(i-1)/n})dA^{\varepsilon}_{r}dW_s,\end{equation}
 \begin{equation}
\label{3.84}
J^{n,\varepsilon,4}:=\sum_{i=1}^{[nt]}\int_{(i-1)/n}^{i/n}\int_{(i-1)/n}^{s-}(Z_{r-}-Z_{(i-1)/n})dM^{\varepsilon}_{r}dW_s,\end{equation}

\begin{equation}
\label{3.85}
J^{n,\varepsilon,5}:=\sum_{i=1}^{[nt]}\int_{(i-1)/n}^{i/n}\int_{(i-1)/n}^{s-}(M^{\varepsilon}_{r-}-M^{\varepsilon}_{(i-1)/n})dW_{r}dA^{\varepsilon}_s,\end{equation}

\begin{equation}
\label{3.86}
J^{n,\varepsilon,6}:=\sum_{i=1}^{[nt]}\int_{(i-1)/n}^{i/n}\int_{(i-1)/n}^{s-}(Z_{r-}-Z_{(i-1)/n})dM^{\varepsilon}_{r}dA^{\varepsilon}_s.\end{equation}

For (\ref{3.81}), since $<M^\varepsilon>_t=\int_{|x|\le \varepsilon}x^2F(dx)t$,    using Doob's inequality,
             $$\lim_{\varepsilon\rightarrow 0}\limsup_{n}\mathbb{E}<nJ^{n,\varepsilon,1}>=0.$$
Then $$\lim_{\varepsilon\rightarrow 0}\limsup_{n}\mathbb{P}(\sup_{t\in [0,1]}|nJ_{t}^{n,\varepsilon,1}|>\rho)=0.$$
For (\ref{3.82}),
    \begin{eqnarray*}
& &\mathbb{E}(\int_{(i-1)/n}^{i/n}\int_{(i-1)/n}^{s-}(M^{\varepsilon}_{r-}-M^{\varepsilon}_{(i-1)/n})dW_{r}dW_s)^2\\
& \le & \int_{(i-1)/n}^{i/n} \mathbb{E} (\int_{(i-1)/n}^{i/n}(M^{\varepsilon}_{r-}-M^{\varepsilon}_{(i-1)/n})dW_{r})^2ds\\
& \le & \int_{(i-1)/n}^{i/n}\int_{(i-1)/n}^{i/n} \mathbb{E} (M^{\varepsilon}_{r-}-M^{\varepsilon}_{(i-1)/n})^2drds\ \end{eqnarray*}
As $\varepsilon \rightarrow 0$,
        $\mathbb{E} (M^{\varepsilon}_{r-}-M^{\varepsilon}_{(i-1)/n})^2\rightarrow 0$, thus
               $$\lim_{\varepsilon\rightarrow 0}\limsup_{n}\mathbb{P}(\sup_{t\in [0,1]}|nJ_{t}^{n,\varepsilon,2}|>\rho)=0.$$

       For (\ref{3.83}) and (\ref{3.84}), both of these two integrals are driven by $W$,  similar to (\ref{3.82}),
       we have         $$\lim_{\varepsilon\rightarrow 0}\limsup_{n}\mathbb{P}(\sup_{t\in [0,1]}|nJ_{t}^{n,\varepsilon,3}+nJ_{t}^{n,\varepsilon,4}|>\rho)=0.$$

        For (\ref{3.85}),     due to  Lemma \ref{local},  by the boundedness of the jumps of $A$ and stationary independent increments of $A$, we  have  $<A^\varepsilon>_t\le Ct$, where $C$ is a constant. and
                         \begin{eqnarray*}
& &\mathbb{E}(\int_{(i-1)/n}^{i/n}\int_{(i-1)/n}^{s-}(M^{\varepsilon}_{r-}-M^{\varepsilon}_{(i-1)/n})dW_{r}dA^{\varepsilon}_s)^2\\
& \le & C\int_{(i-1)/n}^{i/n} \mathbb{E} (\int_{(i-1)/n}^{i/n}(M^{\varepsilon}_{r-}-M^{\varepsilon}_{(i-1)/n})dW_{r})^2ds\\
& \le & C\int_{(i-1)/n}^{i/n}\int_{(i-1)/n}^{i/n} \mathbb{E} (M^{\varepsilon}_{r-}-M^{\varepsilon}_{(i-1)/n})^2drds. \end{eqnarray*}

Following the similar argument for  (\ref{3.86}), we have
             $$\lim_{\varepsilon\rightarrow 0}\limsup_{n}\mathbb{P}(\sup_{t\in [0,1]}|nJ_{t}^{n,\varepsilon,5}+nJ_{t}^{n,\varepsilon,6}|>\rho)=0.$$
             Thus,
                      $$\lim_{\varepsilon\rightarrow 0}\limsup_{n}\mathbb{P}(\sup_{t\in [0,1]}|nK_{t}^{n,\varepsilon}|>\rho)=0.$$

\end{proof}

\subsection{The proof of Theorem \ref{mr1}.}

By the lemmas in the previous subsection, we can obtain the following theorem.

 \begin{theorem} \label{mr3}

 Under Assumption \ref{a1} and $\sigma_s \equiv 1$ for any $s\ge 0$, we have  the following:

 (a) If $c=0$, $nX^n$ weakly converge to 0.

 (b) If $c>0$, $nX^n$ weakly converge to $X$, where
  \begin{eqnarray*}
                 X_s =\frac{\sqrt{6c^3}}{6}M_s-c\sum_{n:T_n\le t}[\sqrt{\xi_n}K_{n}'+\sqrt{c(1-\xi_j)\xi_j}N_j'N_{j}''+\sqrt{1-\xi_n}K_{n}'')]\Delta Z_{T_n}\end{eqnarray*}

\end{theorem}

In this subsection,  we extend Theorem \ref{mr3} to Theorem  \ref {mr1}.

We can construct the processes $\tilde{M}_t^{n,\varepsilon}$, $\tilde{F}_t^{n,\varepsilon}$,  $\tilde{K}_t^{n,\varepsilon}$ through replacing $Z$ by $Y$.

By Assumption 2 and  Lemma \ref{local},  we can easily obtain
we have
       $$\lim_{\varepsilon\rightarrow 0}\limsup_{n}\mathbb{P}(\sup_{t\in [0,1]}|n\tilde{K}_{t}^{n,\varepsilon}|>\rho)=0,$$
     for every $\rho>0$.

  For $\tilde{M}_t^{n,\varepsilon}$,     similar to Lemma \ref{3.5},   we replace $Z_t^c$  by $\int_0^t\sigma_{s-}dZ_s$ in $L_t$, replace $b_{\varepsilon}t$ by $b_{\varepsilon}\int_0^t\sigma_{s-}ds$ in $F_t$, we denote these by $\tilde{L}_t$ and    $\tilde{F}_t$. For $\tilde{q}_i=n
\int_{(i-1)/n}^{i/n}\int_{(i-1)/n}^{s}(\tilde{L}_{r}-\tilde{L}_{(i-1)/n})d\tilde{L}_{r}d\tilde{L}_{s}$.
If we denote $c \cdot \sigma^2_s=c_s$,

    we have
\begin{eqnarray*}
\sum_{i=1}^{[nt]} \mathbb{E}_{i - 1} [\tilde{q}^{2}_{i}] & = & n^{2}
\sum_{i=1}^{[nt]} \mathbb{E}_{i - 1}
\left[\int_{(i-1)/n}^{i/n}\int_{(i-1)/n}^{s}(\tilde{L}_{r}-\tilde{L}_{(i-1)/n})d\tilde{L}_{r}d\tilde{L}_{s}\right]^{2}\\
& = & n^{2} \sum_{i=1}^{[nt]} \mathbb{E}_{i - 1}
\int_{(i-1)/n}^{i/n}\left[\int_{(i-1)/n}^{s}(\tilde{L}_{r}-\tilde{L}_{(i-1)/n})d\tilde{L}_{r}\right]^{2}c_{s-}ds\\
& = & n^{2} \sum_{i=1}^{[nt]} \int_{(i-1)/n}^{i/n}\mathbb{E}_{i -
1}\left[\left(\int_{(i-1)/n}^{s}(\tilde{L}_{r}-\tilde{L}_{(i-1)/n})d\tilde{L}_{r}\right)^{2}c_{s-}\right]ds\\
& = & n^{2} \sum_{i=1}^{[nt]} \int_{(i-1)/n}^{i/n}\mathbb{E}_{i -
1}\left[\left(\int_{(i-1)/n}^{s}(\tilde{L}_{r}-\tilde{L}_{(i-1)/n})d\tilde{L}\right)^{2}c_{(i-1)/n}\right]ds\\
& ~ & + n^{2} \sum_{i=1}^{[nt]} \int_{(i-1)/n}^{i/n}\mathbb{E}_{i -
1}\left[\left(\int_{(i-1)/n}^{s}(\tilde{L}_{r}-\tilde{L}_{(i-1)/n})d\tilde{L}_{r}\right)^{2}(c_{s-}
- c_{(i-1)/n})\right]ds\\
& = & n^{2} \sum_{i=1}^{[nt]} c_{(i-1)/n}
\int_{(i-1)/n}^{i/n}\mathbb{E}_{i -
1}\left[\left(\int_{(i-1)/n}^{s}(\tilde{L}_{r}-\tilde{L}_{(i-1)/n})d\tilde{L}_{r}\right)^{2}\right]ds
+ \epsilon_{1,n}\\
& = & n^{2} \sum_{i=1}^{[nt]} c_{(i-1)/n}
\int_{(i-1)/n}^{i/n}\mathbb{E}_{i -
1}\int_{(i-1)/n}^{s}[(\tilde{L}_{r}-\tilde{L}_{(i-1)/n})^{2} c_{s-}]dr ds +
\epsilon_{1,n}\\
& = & n^{2} \sum_{i=1}^{[nt]} c_{(i-1)/n}
\int_{(i-1)/n}^{i/n}\mathbb{E}_{i -
1}\int_{(i-1)/n}^{s}[(\tilde{L}_{r}-\tilde{L}_{(i-1)/n})^{2} c_{(i-1)/n}]dr ds +
\epsilon_{1,n} + \epsilon_{2,n}\\
& = & n^{2} \sum_{i=1}^{[nt]} c^{2}_{(i-1)/n}
\int_{(i-1)/n}^{i/n}\int_{(i-1)/n}^{s}\mathbb{E}_{i -
1}(\tilde{L}_{r}-\tilde{L}_{(i-1)/n})^{2} dr ds +
\epsilon_{1,n} + \epsilon_{2,n}\\
& = & n^{2} \sum_{i=1}^{[nt]} c^{2}_{(i-1)/n}
\int_{(i-1)/n}^{i/n}\int_{(i-1)/n}^{s}\mathbb{E}_{i -
1}\left(\int_{(i-1)/n}^{r}c_{m} dm \right) dr ds +
\epsilon_{1,n} + \epsilon_{2,n}\\
& = & n^{2} \sum_{i=1}^{[nt]} c^{2}_{(i-1)/n}
\int_{(i-1)/n}^{i/n}\int_{(i-1)/n}^{s}\mathbb{E}_{i -
1}\left(\int_{(i-1)/n}^{r}c_{(i-1)/n} dm \right) dr ds +
\epsilon_{1,n} + \epsilon_{2,n} + \epsilon_{3,n}\\
& = & n^{2} \sum_{i=1}^{[nt]} c^{3}_{(i-1)/n}
\int_{(i-1)/n}^{i/n}\int_{(i-1)/n}^{s}\int_{(i-1)/n}^{r} dm dr ds +
\epsilon_{1,n} + \epsilon_{2,n} + \epsilon_{3,n}\\
& = & n^{2} \sum_{i=1}^{[nt]} c^{3}_{(i-1)/n} \frac{1}{6}
\frac{1}{n^{3}} + \epsilon_{1,n} + \epsilon_{2,n} + \epsilon_{3,n}\\
& = & \frac{1}{6} \int_{0}^{t} c^{3}_{s} ds + \epsilon_{1,n} +
\epsilon_{2,n} + \epsilon_{3,n}.
\end{eqnarray*}
$\epsilon_{1,n} \stackrel{\mathbb{P}} \rightarrow 0,$
$\epsilon_{2,n} \stackrel{\mathbb{P}} \rightarrow 0$ and
$\epsilon_{3,n} \stackrel{\mathbb{P}} \rightarrow 0$ are dealt with
in the similar manner, here we only deal with $\epsilon_{1,n}.$

Under Assumption 1, $\sigma$ is an It\^o semimartingale,  By Lemma 2.1.5 and Lemma 2.1.7 in Jacod and Protter  \cite{jpb},
    \begin{equation}
\label{cc}
\mathbb{E}[\sup_{(i-1)/n\le s\le i/n}|c_{s-}-c_{(i-1)/n}|^p]\le \frac{K}{n}\end{equation}
when $p\ge 2$.  Then

\begin{eqnarray*}
\mathbb{E}[|\epsilon_{1,n}|] & = & \mathbb{E}\left[\left|n^{2}
\sum_{i=1}^{[nt]} \int_{(i-1)/n}^{i/n}\mathbb{E}_{i -
1}\left[\left(\int_{(i-1)/n}^{s}(\tilde{L}_{r}-\tilde{L}_{(i-1)/n})dL_{r}\right)^{2}(c_{s-}
- c_{(i-1)/n})\right]ds\right|\right]\\
& \leq & n^{2} \sum_{i=1}^{[nt]} \int_{(i-1)/n}^{i/n}\mathbb{E}
\left[\left(\int_{(i-1)/n}^{s}(\tilde{L}_{r}-\tilde{L}_{(i-1)/n})dL_{r}\right)^{2}*|c_{s-}
- c_{(i-1)/n}|\right]ds\\
& \leq & n^{2} \sum_{i=1}^{[nt]}
\int_{(i-1)/n}^{i/n}\left[\mathbb{E}
\left(\int_{(i-1)/n}^{s}(\tilde{L}_{r}-\tilde{L}_{(i-1)/n})d\tilde{L}_{r}\right)^{4}\right]^{\frac{1}{2}}*\left[\mathbb{E}|c_{s-}
- c_{(i-1)/n}|^{2}\right]^{\frac{1}{2}} ds\\
 & \leq & K \frac{1}{\sqrt{n}} \cdot n^{2}
\sum_{i=1}^{[nt]} \int_{(i-1)/n}^{i/n}\left[\mathbb{E}
\left(\int_{(i-1)/n}^{s}(\tilde{L}_{r}-\tilde{L}_{(i-1)/n})^{2} c_{r} dr\right)^{2}\right]^{\frac{1}{2}} ds\\
& \leq & K \frac{1}{\sqrt{n}} \cdot n^{2} \sum_{i=1}^{[nt]}
\int_{(i-1)/n}^{i/n}\left[\mathbb{E} \int_{(i-1)/n}^{s}(\tilde{L}_{r}-\tilde{L}_{(i-1)/n})^{4} dr \int_{(i-1)/n}^{s} c^{2}_{r-} dr \right]^{\frac{1}{2}} ds\\
& \leq & K \frac{1}{n} \cdot n^{2} \sum_{i=1}^{[nt]}
\int_{(i-1)/n}^{i/n}\left[\mathbb{E}
\int_{(i-1)/n}^{s}(\tilde{L}_{r}-\tilde{L}_{(i-1)/n})^{4} dr\right]^{\frac{1}{2}}
ds\\
& \leq & K \frac{1}{n} \cdot n^{2} \sum_{i=1}^{[nt]}
\int_{(i-1)/n}^{i/n}\left[\int_{(i-1)/n}^{s}\mathbb{E}
\left(\int_{(i-1)/n}^{r}c^{2}_{r} dt\right)^{2}
dr\right]^{\frac{1}{2}}
ds\\
& \leq & K \frac{1}{n} n^{2} \sum_{i=1}^{[nt]} \frac{1}{n^{5/2}} = K
\frac{1}{\sqrt{n}} t \rightarrow 0
\end{eqnarray*}
 where $K$ denotes a constant.

   Furthermore, if $H = W,$ then

   \begin{eqnarray*}
\sum_{i=1}^{[nt]} \mathbb{E}_{i - 1} [\tilde{q}_{i}\Delta_{i}H] & = & n
\sum_{i=1}^{[nt]} \mathbb{E}_{i - 1} \left[
\int_{(i-1)/n}^{i/n}\int_{(i-1)/n}^{s}(\tilde{L}_{r}-\tilde{L}_{(i-1)/n})d\tilde{L}_{r}d\tilde{L}_{s}
\int_{(i-1)/n}^{i/n} d W_{s} \right]\\
& = & n \sum_{i=1}^{[nt]} \mathbb{E}_{i -
1}\left[\int_{(i-1)/n}^{i/n}\int_{(i-1)/n}^{t}\int_{(i-1)/n}^{s}(\tilde{L}_{r}-\tilde{L}_{(i-1)/n})d\tilde{L}_{r}d\tilde{L}_{s}d
W_{t}\right]\\ & ~ & + n \sum_{i=1}^{[nt]} \mathbb{E}_{i -
1}\left[\int_{(i-1)/n}^{i/n}\int_{(i-1)/n}^{t}W_{s}ds\int_{(i-1)/n}^{t}(\tilde{L}_{r}-\tilde{L}_{(i-1)/n})d\tilde{L}_{r}d\tilde{L}_{t}\right]\\
& ~ & + n \sum_{i=1}^{[nt]} \mathbb{E}_{i -
1}\left[\int_{(i-1)/n}^{i/n}\int_{(i-1)/n}^{t}(\tilde{L}_{r}-\tilde{L}_{(i-1)/n})d\tilde{L}_{r} \sigma_{s-} dt\right]\\
& = & n \sum_{i=1}^{[nt]} \mathbb{E}_{i -
1}\left[\int_{(i-1)/n}^{i/n}\int_{(i-1)/n}^{t}(\tilde{L}_{r}-\tilde{L}_{(i-1)/n})d\tilde{L}_{r}
\sigma_{(i - 1)/n} dt\right] + \epsilon_{4,n}\\ & = &
\epsilon_{4,n},
\end{eqnarray*}
by the Fubini's Theorem for the third part and the martingale
property of stochastic integral.

Finally, we prove that $\epsilon_{4,n} \stackrel{\mathbb{P}}
\rightarrow 0.$
\begin{eqnarray*}
& ~ & \mathbb{E}[|\epsilon_{4,n}|]\\
 & \leq & n \sum_{i=1}^{[nt]} \mathbb{E}
\left[\int_{(i-1)/n}^{i/n}|\int_{(i-1)/n}^{t}(\tilde{L}_{r}-\tilde{L}_{(i-1)/n})dL_{r}||\sigma_{s-}
- \sigma_{(i - 1)/n}|dt\right]\\
& \leq & n \sum_{i=1}^{[nt]}
\int_{(i-1)/n}^{i/n}\left[\mathbb{E}\left(\int_{(i-1)/n}^{t}(\tilde{L}_{r}-\tilde{L}_{(i-1)/n})d\tilde{L}_{r}\right)^{2}\right]^{\frac{1}{2}}\left[\mathbb{E}(\sigma_{s-}
- \sigma_{(i - 1)/n})^{2}\right]^{\frac{1}{2}}dt\\
& = & \frac{1}{\sqrt{n}} \cdot n \sum_{i=1}^{[nt]} \int_{(i-1)/n}^{i/n} \left[\mathbb{E}\int_{(i-1)/n}^{t}(\tilde{L}_{r}-\tilde{L}_{(i-1)/n})^{2} c_{r-} dr\right]^{\frac{1}{2}}dt\\
& \leq & K \frac{1}{\sqrt{n}} \cdot n \sum_{i=1}^{[nt]}
\int_{(i-1)/n}^{i/n} \left[\int_{(i-1)/n}^{t}(r - (i-1)/n) dr
\right]^{\frac{1}{2}}dt\\
& \leq & K \frac{1}{\sqrt{n}} t \rightarrow 0,
\end{eqnarray*}
by Cauchy-Schwarz inequality and (\ref{cc}) where $K$
denotes a constant.

For $\tilde{F}^{n,\varepsilon}$ part,  it can be obtained the similar procedure and Theorem \ref{kp1}. We omit it.

\section{Application and discussions}
\subsection{Application}
Now, we discuss the approximation of  Dol\'{e}ans-Dade exponential.  We present numerical method to solve
   \begin{equation*}   dX_t=X_{t-}dY_t.
\end{equation*}
We introduce the Milstein type method:
               \begin{eqnarray*}
                   X_t^n&=& X_{n(t)}^{n}+X^{n}_{n(t)}(Y_{t}-Y_{n(t)})\\
                        & & +X^{n}_{n(t)}\int_{n(t)}^{t}(Y_{s-}-Y_{n(s)})dY_s.\end{eqnarray*}
    where $n(t)=k/n$, if $k/n<t\le (k+1)/n$.     We want to study the weak convergence of $U^n$, where
                      $$U^n_t=X^n_{[nt]}-X_{[nt]}.$$

             We have the following theorem.

             \begin{theorem}\label{mr2}
             Under Assumption \ref{a1}, we have  the following:

 (a) If $c=0$, $nU^n$ weakly converge to 0.

 (b) If $c>0$, $nU^n$ weakly converge to $U$, where $U$ is the unique solution of the following linear equation:
   \begin{eqnarray*}
                 U_t&=&\int_{0}^{t}U_{s-}dY_s-\frac{\sqrt{6c^3}}{6}\int_{0}^{t}X_{s-}\sigma_{s-}^{3}dM_s\\
                       & &-c\sum_{n:T_n\le t}[\sqrt{\xi_n}K_{n}'(X_{T_{n}-})^2+(\sqrt{c(1-\xi_n)\xi_n}N_n'N_{n}''+\sqrt{1-\xi_n}K_{n}'')(X_{T_{n-}})^2]\sigma_{T_{n}-}^3\Delta Z_{T_n}.
                       \end{eqnarray*}
           \end{theorem}

           Before proving this theorem, we need the following theorem, which can help us to make connection Theorem  \ref{mr1} and \ref{mr2}.

Consider
$$X^n_t=J^n_t+\int_0^tX_{s-}^nH_{s}^ndY_s$$
 where $Y$ is a given semimartingale, $(J^n)_{n\ge 1}$ is a sequence of adapted c\'{a}dl\'{a}g processes and $(H^n)_{n\ge 1}$ is a sequence of predictable processes.

     \begin{theorem} (Jacod and Protter \cite{jp}.)     Let  $V_t^n=\int_0^tH_s^ndY_s$ . Suppose $\sup_{s\in [0,1]}|H_s^n|$  is tight, and
           $$(J^n,V^n,\rho^n)\stackrel{stably}{\Rightarrow} (J,V,\rho)$$
   on some extension of the space. Then $V$ is a semimartingale on the extension, and
               $$(J^n,V^n,X^n,\rho^n)\stackrel{stably}{\Rightarrow} (J,V,X,\rho)$$
   where $X$ is the unique solution of
   $$X_t=J_t+\int_0^tX_{s-}dV_s.$$    \end{theorem}.

           \begin{proof}

        In fact,
 \begin{eqnarray*}
 U^{n}_{t}&=& \sum_{i=1}^{[nt]}\int_{(i-1)/n}^{i/n}X_{(i-1)/n}^{n}dY_{s}\\
 & &+\sum_{i=1}^{[nt]}X_{(i-1)/n}^{n}\int_{(i-1)/n}^{i/n}(Y_{s-}-Y_{(i-1)/n})dY_{s}-\sum_{i=1}^{[nt]}\int_{(i-1)/n}^{i/n}X_{s-}dY_{s}\\
 &=&  \sum_{i=1}^{[nt]}\int_{(i-1)/n}^{i/n}[X_{s-}^{n}-X_{s-}]dY_{s}-\sum_{i=1}^{[nt]}\int_{(i-1)/n}^{i/n}[X_{s-}^{n}-X_{(i-1)/n}^{n}]dY_{s}\\
 & & +\sum_{i=1}^{[nt]}X_{(i-1)/n}^{n}\int_{(i-1)/n}^{i/n}(Y_{s-}-Y_{(i-1)/n})dY_{s} \end{eqnarray*}
 and

 \begin{eqnarray*}
 & & \sum_{i=1}^{[nt]}X_{(i-1)/n}^{n}\int_{(i-1)/n}^{i/n}(Y_{s-}-Y_{(i-1)/n})dY_{s}-\sum_{i=1}^{[nt]}\int_{(i-1)/n}^{i/n}(X_{s-}^{n}-X_{(i-1)/n}^{n})dY_{s}\\
 &=&-\sum_{i=1}^{[nt]}X_{(i-1)/n}^{n}\int_{(i-1)/n}^{i/n}\int_{(i-1)/n}^{s-}(Y_{r-}-Y_{(i-1)/n})dY_{r}dY_{s} .\end{eqnarray*}

      Thus
       \begin{eqnarray*}
 U^{n}_{t}&=&\sum_{i=1}^{[nt]}\int_{(i-1)/n}^{i/n}[X_{s-}^{n}-X_{s-}]dY_{s}\\
 & &-\sum_{i=1}^{[nt]}X_{(i-1)/n}^{n}\int_{(i-1)/n}^{i/n}\int_{(i-1)/n}^{s-}(Y_{r-}-Y_{(i-1)/n})dY_{r}dY_{s} . \end{eqnarray*}

 Set
    $$R_t^{n,\varepsilon}=\int_{0}^{t}X_sd(M^{n,\varepsilon}+F^{n,\varepsilon})_s,$$

     Introduce the following equation,
          \begin{eqnarray*}
 dV^{n,\varepsilon}_{s}&=&V^{n,\varepsilon}_{s-}dZ_{s}-dR_s^{n,\varepsilon}.\end{eqnarray*}

By Lemma 2.4 in Jacod and Protter \cite{jp},

  \begin{eqnarray*}
  \mathbb{P}(\sup_{0\le t\le 1}|V^{n,\varepsilon}_{t}- V^{n}_{t}|>\rho)&\le& \rho'+ \mathbb{P}(\sup_{0\le t\le 1}|X_t|>A_1)\\
  & &+\mathbb{P}(\sup_{0\le t\le 1}|R_t^{n,\varepsilon}|>A_2)+\mathbb{P}(\sup_{0\le t\le 1}|R_t^n-R_t^{n,\varepsilon}|>\omega_1)+\frac{\omega_1}{\rho}K_{A_1,\rho'}.\end{eqnarray*}
      By Lemma \ref{3.5}, \ref{3.7}, \ref{3.8},  we can obtain
           $$\lim_{\varepsilon\rightarrow 0}\limsup_{n} \mathbb{P}(\sup_{0\le t\le 1}|V^{n,\varepsilon}_{t}- V^{n}_{t}|>\rho)=0.$$
      By the weak convergence of stochastic integral and stability of stochastic differential equations,
     the limiting processes of $V^{n,\varepsilon}$ is the solution of
       \begin{eqnarray*}
                 V_t&=&\int_{0}^{t}V_{s-}dZ_s-\frac{\sqrt{6c^3}}{6}\int_{0}^{t}X_{s-}dM_s\\
                       & &-c\sum_{n:T_n\le t}[\sqrt{\xi_n}K_{n}'X_{T_{n}-}+(\sqrt{c(1-\xi_j)\xi_j}N_j'N_{j}''+\sqrt{1-\xi_n}K_{n}'')X_{T_{n-}}^2]\Delta Z_{T_n}
                     \end{eqnarray*}
     where  $M$ is standard Brownian motion, which are independent from $Z^c$,
               The final version of Theorem \ref{mr1} can be extended from $V^{n,\varepsilon}$ and $V$ via the discretization  and weak convergence of  stochastic integrals.

  \end{proof}

  \subsection{Discussion}
  When we consider  a general stochastic differential equation (SDE) with the form:
             \begin{equation}\label{sde}
                   X_t=x_0+\int_{0}^{t}f(X_s)dY_s,\end{equation}
where $f$ denotes a $C^3$ (three times differentiable) function, and $Y$ is semimartingale.    We solve this SDE numerically by means of the Milstein method,
  \begin{eqnarray*}
                   X_t^n&=& X_{n(t)}^{n}+f(X^{n}_{n(t)})(Y_{t}-Y_{n(t)})\\
                        & & +f(X^{n}_{n(t)})f'(X^{n}_{n(t)})\int_{n(t)}^{t}(Y_{s-}-Y_{n(s)})dY_s.\end{eqnarray*}
                We want to study the weak convergence of
                   $$U^n_t=X^n_{[nt]}-X_{[nt]}.$$
                   Similar to the previous study,
                          \begin{eqnarray*}
 U^{n}_{t}&=& \sum_{i=1}^{[nt]}\int_{(i-1)/n}^{i/n}f(X_{(i-1)/n}^{n})dY_{s}\\
 & &+\sum_{i=1}^{[nt]}f'(X_{(i-1)/n}^{n})f(X_{(i-1)/n}^{n})\int_{(i-1)/n}^{i/n}(Y_{s-}-Y_{(i-1)/n})dY_{s}\\
 & &-\sum_{i=1}^{[nt]}\int_{(i-1)/n}^{i/n}f(X_{s-})dY_{s}\\
 &=&  \sum_{i=1}^{[nt]}\int_{(i-1)/n}^{i/n}[f(X_{s-}^{n})-f(X_{s-})]dY_{s}-\sum_{i=1}^{[nt]}\int_{(i-1)/n}^{i/n}[f(X_{s-}^{n})-f(X_{(i-1)/n}^{n})]dY_{s}\\
 & & +\sum_{i=1}^{[nt]}f'(X_{(i-1)/n}^{n})f(X_{(i-1)/n}^{n})\int_{(i-1)/n}^{i/n}(Y_{s-}-Y_{(i-1)/n})dY_{s}\\
 &=&  \sum_{i=1}^{[nt]}\int_{(i-1)/n}^{i/n}[f(X_{s-}^{n})-f(X_{s-})]dY_{s}\\
 & & + \sum_{i=1}^{[nt]}f'(X_{(i-1)/n}^{n})f(X_{(i-1)/n}^{n})\int_{(i-1)/n}^{i/n}(Y_{s-}-Y_{(i-1)/n})dY_{s}\\
 & &-\sum_{i=1}^{[nt]}\int_{(i-1)/n}^{i/n}f'(X_{(i-1)/n}^{n})(X_{s-}^{n}-X_{(i-1)/n}^{n})dY_{s}\\
 & &-\frac{1}{2}\sum_{i=1}^{[nt]}\int_{(i-1)/n}^{i/n}f''(\tilde{X}_{(i-1)/n}^{n})(X_{s-}^{n}-X_{(i-1)/n}^{n})^2dY_{s}
 \end{eqnarray*}
 where $\tilde{X}_{(i-1)/n}^{n}$ is  random variable between $X_{(i-1)/n}^{n}$  and  $X_{s}^{n}$.
    When $f''(x)\neq 0$, we need to study the asymptotic properties of

             \begin{eqnarray*}
 & &\sum_{i=1}^{[nt]}\int_{(i-1)/n}^{i/n}f''(\tilde{X}_{(i-1)/n}^{n})(X_{s-}^{n}-X_{(i-1)/n}^{n})^2dY_{s}\\
 &=& \sum_{i=1}^{[nt]}f''(\tilde{X}_{(i-1)/n}^{n})(f(X_{(i-1)/n}^{n}))^{2}\int_{(i-1)/n}^{i/n}(Y_{s-}-Y_{(i-1)/n})^2dY_{s}\\
 & &+ \sum_{i=1}^{[nt]}f''(\tilde{X}_{(i-1)/n}^{n})(f(X_{(i-1)/n}^{n})f'(X_{(i-1)/n}^{n}))^{2}\int_{(i-1)/n}^{i/n}(\int_{(i-1)/n}^{s-}(Y_{r-}-Y_{(i-1)/n})dY_{r})^2dY_{s}.\end{eqnarray*}

 Similar to the previous discussion, the weak convergence of
        $$n\sum_{i=1}^{[nt]}\int_{(i-1)/n}^{i/n}(Z_{s-}-Z_{(i-1)/n})^2dZ_{s},$$
  is important.
   it has to discuss the weak convergence of
       $$\sum_{T_j\le t}n(\Delta Z_{T_j})^2(Z_{T_+(n,j)}^{c}-Z_{T_j}^{c}).$$
Unfortunately,    when the normalized rate is $n$, this term does not converges weakly.  In the future, we will study the rate of convergence of this term.

    \bibliographystyle{amsplain}

 \end{document}